\documentclass[11pt,reqno]{amsart} 
\usepackage{amsmath,tikz}
\usepackage[utf8]{inputenc}
\usepackage{hyperref}
\usepackage[normalem]{ulem}
\usepackage{amsmath}
\usepackage{amssymb}
\usepackage{amsthm}
\usepackage{mathtools}
\usepackage{verbatim}
\usepackage{xcolor}
\usepackage{fullpage}
\usepackage[skip=3pt, indent=20pt]{parskip}
\usepackage{caption}
\usepackage{comment}
\usepackage{marginnote}

\theoremstyle{plain}
\newtheorem{thm}{Theorem}[section]
\newtheorem{lem}[thm]{Lemma}

\newtheorem{cor}[thm]{Corollary}

\theoremstyle{definition}

\newtheorem{defin}[thm]{Definition}
\newtheorem{ex}[thm]{Example}

\theoremstyle{remark}

\newcommand{\Mesa}{\text{Mesa}}

\newcommand\AMS{\mathsf{AMS}}

\title{Mesas of Stirling Permutations}
\author{Nicolle Gonz\'{a}lez}
\address[N.~Gonz\'{a}lez]{Department of Mathematics, University of California, Berkeley, CA,  94720}
\email{\textcolor{blue}{\href{mailto:nicolle@math.berkeley.edu}{nicolle@math.berkeley.edu}}}
\thanks{}

\author{Pamela E. Harris}
\address[P.~E.~Harris]{Department of Mathematical Sciences, University of Wisconsin, Milwaukee, WI 53211}
\email{\textcolor{blue}{\href{mailto:peharris@uwm.edu}{peharris@uwm.edu}}}
\thanks{P.E.H. was partially supported through a Karen Uhlenbeck EDGE Fellowship.}

\author{Gordon Rojas Kirby}
\address[G.~Rojas Kirby]{Department of Mathematics and Statistics, San Diego State University, San Diego, CA 92182}
\email{\textcolor{blue}{\href{mailto:gkirby@sdsu.edu}{gkirby@sdsu.edu}}}
\thanks{}
 
\author{Mariana Smit Vega Garcia}
\address[M.~Smit Vega Garcia]{Department of Mathematics, Western Washington University, Bellingham, WA 98225}
\email{\textcolor{blue}{\href{mailto:}{smitvem@wwu.edu}}}
\thanks{M.S.V.G was partially supported by the NSF grant DMS-2054282.}

\author{Bridget Eileen Tenner}
\address[B.~E.~Tenner]{Department of Mathematical Sciences, DePaul University, Chicago, IL 60614}
\email{\textcolor{blue}{\href{mailto:bridget@math.depaul.edu}{bridget@math.depaul.edu}}}
\thanks{B.E.T. was partially supported by the NSF grant DMS-2054436.}

\begin{document}

\maketitle
\begin{abstract}
Given a Stirling permutation $w$, we introduce the mesa set of $w$ as the natural generalization of the 
pinnacle set of a permutation. Our main results characterize admissible mesa sets 
and give closed enumerative formulas in terms of rational Catalan numbers by providing an explicit bijection between mesa sets and rational Dyck paths. 

\end{abstract}

\section{Introduction}\label{sec:intro}

A \emph{Stirling permutation of order $n$} is a word $w = w(1)w(2)\cdots w(2n)$ consisting of the multiset of letters $\{1,1,\ldots, n,n\}$, with the property that for each $1\leq k \leq n$, all of the values appearing between the two copies of $k$ in $w$ are larger than $k$. 
We write $Q_n$ for the set of Stirling permutations of order $n$. For instance, $884425536776321199 \in Q_9$, illustrated in Figure~\ref{fig:newfig}, while $31324421 \not\in Q_4$
 because the value $1$ appears between the two copies of $3$. 
Stirling permutations were introduced by Gessel and Stanley in their study of Stirling numbers of the first and second kind \cite{GesselStanley}. Descent statistics and connections to rooted trees for these permutations and their generalizations have been well studied in recent years \cite{Bona,Damir, Elizalde,Janson, JansonBona, KubaVarvak}.

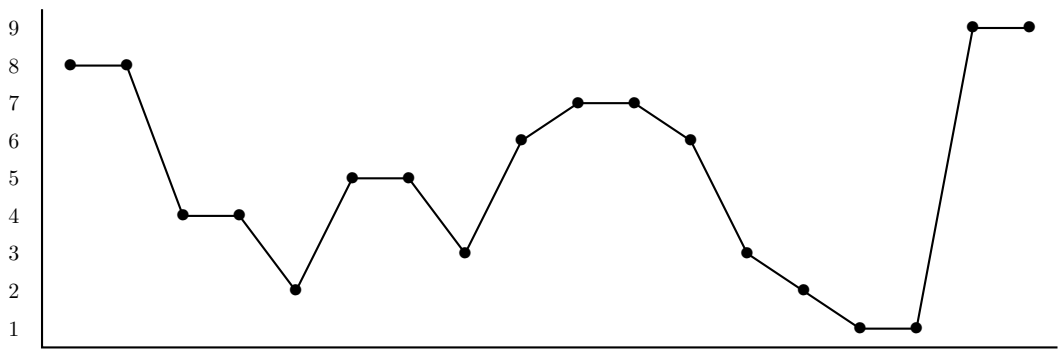
\begin{figure}[htbp]
\begin{tikzpicture}[yscale=.5, xscale=.75]
\draw[thick] (0.5,9.5)--(0.5,0.5)--(18.5,0.5);
\node at (0,1)[scale=.75]{$1$};
\node at (0,2)[scale=.75]{$2$};
\node at (0,3)[scale=.75]{$3$};
\node at (0,4)[scale=.75]{$4$};
\node at (0,5)[scale=.75]{$5$};
\node at (0,6)[scale=.75]{$6$};
\node at (0,7)[scale=.75]{$7$};
\node at (0,8)[scale=.75]{$8$};
\node at (0,9)[scale=.75]{$9$};
\draw[thick] (1,8)--(2,8)--(3,4)--(4,4)--(5,2)--(6,5)--(7,5)--(8,3)--(9,6)--(10,7)--(11,7)--(12,6)--(13,3)--(14,2)--(15,1)--(16,1)--(17,9)--(18,9);
\node at (1,8){$\bullet$};
\node at (2,8){$\bullet$};
\node at (3,4){$\bullet$};
\node at (4,4){$\bullet$};
\node at (5,2){$\bullet$};
\node at (6,5){$\bullet$};
\node at (7,5){$\bullet$};
\node at (8,3){$\bullet$};
\node at (9,6){$\bullet$};
\node at (10,7){$\bullet$};
\node at (11,7){$\bullet$};
\node at (12,6){$\bullet$};
\node at (13,3){$\bullet$};
\node at (14,2){$\bullet$};
\node at (15,1){$\bullet$};
\node at (16,1){$\bullet$};
\node at (17,9){$\bullet$};
\node at (18,9){$\bullet$};
\end{tikzpicture} 
\caption{The graph of the Stirling permutation $884425536776321199 \in Q_9$.}\label{fig:newfig}
\end{figure}

\emph{Pinnacle sets} of permutations have received great interest since their introduction in \cite{PinnaclesTypeA} by Davis, Nelson, Petersen, and Tenner. For the classical symmetric group, several properties of these sets have been studied \cite{DHHIN,Domagalski, FNT,fang, Minnich,rusu,rusu-tenner}. More recently, the authors of this article characterized and enumerated pinnacle sets for signed permutations \cite{PinnaclesTypeBD}. Given the depth and beauty in these previous results, it is natural to consider pinnacles in Stirling permutations. However, as we quickly show in Lemma~\ref{lem:stirling perms have no traditional peaks}, Stirling permutations do not have pinnacles, in the traditional sense. To accommodate this, we expand the idea of a ``pinnacle'' in the context of Stirling permutations: given that, topographically, mesas are hill formations that have cliffs on all sides, we use ``mesas'' to describe the analogue of pinnacles in the setting of Stirling permutations. 
That is, if there exists an index $i\in \{2,\ldots,2n-2\}$, such that $w(i-1)<w(i)=w(i+1)>w(i+2)$, then we call the value $w(i)$ a \textit{mesa} of $w$. 
Let $[n]\coloneqq \{1,\dots,n\}$. 
Then, given a Stirling permutation $w$, we write 
$$\Mesa(w)\coloneqq\{k\in[n]: k\mbox{ is a mesa of $w$}\}.$$
For example, recalling the Stirling permutation illustrated in Figure~\ref{fig:newfig}, we have 
$$\Mesa(884425536776321199)=\{5,7\}.$$

Analogously, instead of using the term \emph{vale} from the classical and signed settings, we expand this notion and instead consider the \emph{local minima} of a Stirling permutation. More precisely, the local minima of a Stirling permutation $w \in Q_n$ are the values $w(i)$ that are smaller than their nearest non-equal neighbors (if any) in $w$. In particular, if $w(1)$ is less than its leftmost non-equal neighbor, then $w(1)$ is a local minimum of $w$. Note that these are exactly the local minima of the graph of the permutation, in the usual sense. 
For example, the local minima of the Stirling permutation illustrated in Figure~\ref{fig:newfig} are $\{1,2,3\}$. 

Our main interest is to determine the subsets of $[n]$ that can be the mesa sets of Stirling permutations. To this end, we give the following definition.

\begin{defin}\label{defn:admissible}
A set $M\subset [n]$ is an \emph{admissible mesa set} if there exists a Stirling permutation $w \in Q_n$ such that $\Mesa(w) = M$. Any such $w$ is a \emph{witness} for $M$. We write $\AMS_n$ to denote the set of admissible mesa sets for Stirling permutations of order $n$, and $\left|\AMS_n\right|$ for the number of admissible mesa sets among all Stirling permutations of order $n$.
\end{defin}

Note that if $M$ is an admissible mesa set in $Q_n$, then in fact $M$ is an admissible mesa set in $Q_{n'}$ for any $n' \ge \max\{M\}$. This follows from repeated application of Lemma~\ref{lem:ams of order n vs of order n+1}(a), proved below. Hence, we can refer to an admissible mesa set without specifying the order of a witness. 

\begin{ex} \label{ex}
 For all positive integers $n$, the Stirling permutation $1122\cdots nn \in Q_n$ is a witness for the admissible mesa set $\emptyset$. The Stirling permutation $1334664225518877 \in Q_8$ is a witness for the admissible mesa set $\{5,6,8\}$, and this Stirling permutation is depicted in Figure~\ref{fig:ex3}.

\begin{figure}[htpb]
\center
\begin{tikzpicture}[yscale=.5, xscale=.75]
\draw[thick] (0.5,8.5)--(0.5,0.5)--(16.5,0.5);

\node at (0,1)[scale=.75]{$1$};
\node at (0,2)[scale=.75]{$2$};
\node at (0,3)[scale=.75]{$3$};
\node at (0,4)[scale=.75]{$4$};
\node at (0,5)[scale=.75]{$5$};
\node at (0,6)[scale=.75]{$6$};
\node at (0,7)[scale=.75]{$7$};
\node at (0,8)[scale=.75]{$8$};
\draw[thick] (1,1)--(2,3)--(3,3)--(4,4)--(5,6)--(6,6)--(7,4)--(8,2)--(9,2)--(10,5)--(11,5)--(12,1)--(13,8)--(14,8)--(15,7)--(16,7);
\node at (1,1){$\bullet$};
\node at (2,3){$\bullet$};
\node at (3,3){$\bullet$};
\node at (4,4){$\bullet$};
\node at (5,6){$\bullet$};
\node at (6,6){$\bullet$};
\node at (7,4){$\bullet$};
\node at (8,2){$\bullet$};
\node at (9,2){$\bullet$};
\node at (10,5){$\bullet$};
\node at (11,5){$\bullet$};
\node at (12,1){$\bullet$};
\node at (13,8){$\bullet$};
\node at (14,8){$\bullet$};
\node at (15,7){$\bullet$};
\node at (16,7){$\bullet$};
\end{tikzpicture} 
\caption{The graph of the Stirling permutation $1 33 4 66 4 22 55 1 88 77 \in Q_8$.\label{fig:ex3}}
\end{figure}
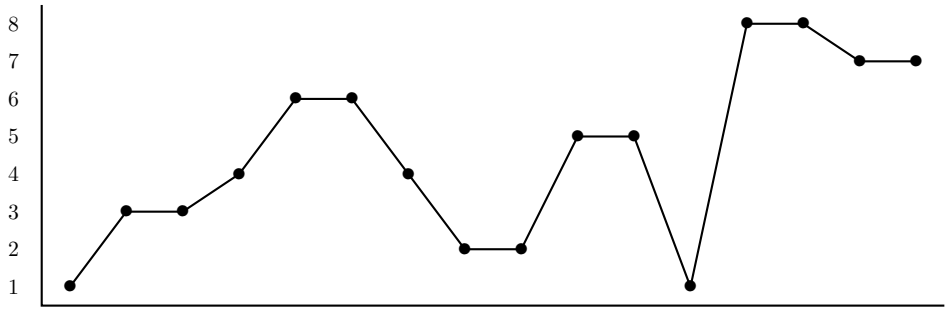
\end{ex}

Our first main result, presented in Theorem~\ref{thm:characterization}, characterizes admissible mesa sets for Stirling permutations. Specifically, we prove that a set $M$ of positive integers is an admissible mesa set if and only if
$$\Big| M \cap [1,x] \Big|  \le \frac{2x-1}{3}$$
for all $x \in M$. 

In order to count these sets, we then define a bijection between maximal mesa sets for $n=3k-1$ and certain rational Dyck paths, which, combined with certain tri-cyclic behaviour of mesa sets, yields the desired enumeration result. Namely, Corollary~\ref{cor:complete enumeration} states that for $n = 3k + r$, where $r \in \{0,1,2\}$ and $k$ is a non-negative integer, we have 
\[|\AMS_n| = 2^{n-1} - \sum_{i=0}^{k-1}2^{3i+r}C_{2(k-i)-1, k-i},\]
where
\[C_{m,\ell} = \frac{1}{\ell}\binom{\ell+m-1}{m}\]
corresponds to the \emph{rational Catalan number} \cite[\href{https://oeis.org/A328901}{A328901}]{OEIS}.

This article is organized as follows. In Section~\ref{sec:prelims}, we provide some notation and initial results. 
In Section~\ref{sec:characterization}, we prove Theorem~\ref{thm:characterization} and use these results in Section~\ref{sec:enumerations} to give a proof of our main enumerative result in Corollary~\ref{cor:complete enumeration}. 
We conclude with Section~\ref{sec:future}, describing directions for possible future research.

\section{Definitions and preliminary results}\label{sec:prelims}

Following the notation of Gessel and Stanley \cite[Theorem 2.1]{GesselStanley}, we let
$Q_n$
denote the set of Stirling permutations of order $n$.
As described in the introduction, we write Stirling permutations as words $w(1)\cdots w(2n) \in Q_n$, where $w(i)\in[n]$ for all $1\leq i\leq 2n$. 
For example, $Q_2=\{1122,2211,1221\}$ and 
\[Q_3=\left\{\begin{matrix}

     112233, &  122133, &  221133, &112332, &  122331,\\
221331, &113322, &  123321,  & 223311, &133122,   \\
233211, &331122, &  331221,  & 332211, &133221\phantom{,}
\end{matrix}\right\}     .\]
It is a direct consequence of this definition that there are no 
pinnacles  (values $w(i)$ satisfying $w(i-1)<w(i)>w(i+1)$) in Stirling permutations. 

\begin{lem}\label{lem:stirling perms have no traditional peaks}
For any Stirling permutation $w = w(1)\cdots w(2n) \in Q_n$, there is no $i\in[2,2n-1]$ for which $w(i-1) < w(i) > w(i+1)$. 
\end{lem}

\begin{proof}
Suppose, for the purpose of obtaining a contradiction, that there is such an $i$. Then the two appearances of $w(i)$ are not consecutive in $w$, and thus either $w(i-1)$ or $w(i+1)$ appears between them. Because $w(i)$ is larger than both of those values, this would mean that $w \not\in Q_n$.
\end{proof}

In particular, it follows immediately that each local maximum in a given Stirling permutation must appear twice consecutively in the word. 

Our goal in this work is to study the set $\AMS_n = \{\Mesa(w): w \in Q_n\}$ of admissible mesa sets. As we do this, certain properties of admissible mesa sets become immediately apparent. For example, there is a natural bound on the number of mesas that a Stirling permutation can have.

\begin{lem}\label{L:most}
A Stirling permutation of order $n$ can have at most $(2n-1)/3$ mesas. Moreover, this bound is sharp: there exists $w \in Q_n$ with $|\Mesa(w)| = \lfloor (2n-1)/3 \rfloor$.
\end{lem}

\begin{proof} 
Fix $w \in Q_n$, and suppose that $|\Mesa(w)| = k$. In the one-line notation of $w$, each mesa appears twice consecutively, and this pair must be surrounded by smaller values. Thus there are at least $k+1$ local minima in $w$. Because $w$ is a Stirling permutation, these local minima can each be repeated once. Therefore, we need
$$n \ge k + \frac{k+1}{2},$$
and the result follows. 

Setting $k \coloneqq\lfloor (2n-1)/3 \rfloor$, the Stirling permutation
$$t_1 \ m_1 \ m_1 \ t_2\ m_2 \ m_2 \ t_3 \ \cdots \ t_k \ m_k \ m_k \ t_{k+1} \ t_{k+2}\ t_{k+3},$$
where $m_i=n-k+i$ and $t_i = \lceil i/2 \rceil$ whenever $\lceil i/2 \rceil \leq n-k$, and with $t_i =\emptyset$ otherwise, has $k$~mesas.
\end{proof}

Lemma~\ref{L:most} demonstrates a sort of tri-cyclic behavior to mesa sets, illustrated in the following example.

\begin{ex}
The following Stirling permutations each have maximally many mesas:
\begin{align*}
    &144155266233 \in Q_6, \\
    &14415526627733 \in Q_7, \text{ and}\\
    &1441552662773883 \in Q_8.
\end{align*}
\end{ex}

While there can be multiple witness Stirling permutations for a given admissible mesa set, we 
often refer to a particular witness Stirling permutation that we call ``canonical.'' This will mimic the terminology used in
\cite{PinnaclesTypeA,Domagalski,PinnaclesTypeBD}.

\begin{defin} \label{defn:canonical witness}
For $M\in \AMS_n$, write $M \coloneqq\{m_1<\cdots<m_l\}$, and set 
$\overline{M}\coloneqq[n]\setminus M =\{u_1<\cdots<u_{n-l}\}$. 
The \emph{canonical witness permutation} for $M$ is 
\begin{equation}\label{eqn:canonical witness}
t_1 \ m_1 \ m_1 \ t_2 \ m_2 \ m_2 \ \cdots \ t_l \ m_l \ m_l \ t_{l+1}\ \cdots \ t_{2(n-l)} \in Q_n,
\end{equation}
where $t_j=u_{\lceil j/2\rceil}$.
\end{defin}

More expansively, if $l = 2p$, then the canonical witness for $M$ is
\[
u_1 \ m_1 \ m_1 \ u_1 \ m_2 \ m_2 \ u_2 \ \cdots \ u_{p}  \ m_{2p} \ m_{2p} \ u_{p+1} \ u_{p+1} \ \cdots \ u_{n-l} \ u_{n-l},
\]
and if $l=2p+1$, then the canonical witness for $M$ is
\[
u_1 \ m_1 \ m_1 \ u_1 \ m_2 \ m_2 \ u_2 \ \cdots \ u_{p+1}  \ m_{2p+1} \ m_{2p+1} \ u_{p+1} \ u_{p+2} \ u_{p+2} \ \cdots \ u_{n-l} \ u_{n-l}.
\]

\begin{ex}\label{ex:bigmesa}
The admissible mesa set
$M = \{5,6,8\} \in \AMS_8$ 
has canonical witness permutation 
$15516628823377 \in Q_8$. 
Notice that this is a different permutation than the witness given for this mesa set in Figure~\ref{fig:ex3}. 
\end{ex}

\section{Characterization}\label{sec:characterization}

Our study of admissible mesa sets for Stirling permutations is both characterizing and enumerating. We begin with the characterization, which is given in the same language as \cite[Theorem~3.4]{rusu-tenner}. Throughout this section, we let $M$ be an admissible mesa set  and we set the following notation for any positive integer $x$:
$$M_x \coloneqq [1,x] \cap M.$$

\begin{ex}
   As demonstrated in Figure~\ref{fig:newfig}, the set $M=\{5,7\}$ is an admissible mesa set. For this $M$, we have $M_6 = M_5 =\{5\}$. 
\end{ex}

We can now characterize admissible mesa sets. Note that when $x$ is the largest mesa, this generalizes Lemma~\ref{L:most}.

\begin{thm}\label{thm:characterization}
A set $M$ of positive integers is an admissible mesa set if and only if
\begin{equation}\label{eqn:M_x bound}
|M_x| \le \frac{2x-1}{3}
\end{equation}
for all $x \in M$. 
\end{thm}

\begin{proof}
Suppose, first, that inequality~\eqref{eqn:M_x bound} holds for all $x \in M$. This family of assumptions mean that the (canonical witness) permutation described in expression~\eqref{eqn:canonical witness} has mesa set $M$.

On the other hand, suppose that $M$ is an admissible mesa set. The smallest $|M_x|$ mesas in $M$ must be surrounded by at least $|M_x| + 1$ local minima, and local minima are necessarily smaller than their neighboring mesas. 
The fact that we are working with Stirling permutations, in which each value may appear twice, means that there are $2(x - |M_x|)$ values available to serve as these local minima. Thus, $|M_x|+1\leq 2(x - |M_x|)$. Thus
$|M_x| \le (2x-1)/3$ for all $x \in M$.
\end{proof}

Theorem~\ref{thm:characterization} yields an interesting collection of properties, beginning with a characterization of the ``minimal'' admissible mesa set.

\begin{cor}\label{cor:minimal mesa set}
Let $M = \{m_1 < m_2 < \cdots < m_l\}$ be an admissible mesa set. Then $(m_1,m_2,\ldots, m_l)$ must be at least as large, component-wise, as
\begin{equation}\label{eqn:minimal mesas}
\left(2,4,5,7,8,10,11,\ldots,\left\lfloor\frac{3l}{2}\right\rfloor+1\right).
\end{equation}
\end{cor}

\begin{proof}
Suppose, for the purpose of obtaining a contradiction, that this is not the case, and let $i$ be minimal so that $m_i \le \lfloor 3i/2\rfloor$. Certainly $i > 1$, because $1$ cannot be a mesa. Thus the sequence given in \eqref{eqn:minimal mesas} and the minimality of $i$ mean that $i$ must be even. Then
$$\frac{2m_i - 1}{3} \le \frac{2 \left\lfloor \frac{3i}{2} \right\rfloor - 1}{3} = \frac{2 \cdot \frac{3i}{2} - 1}{3} = i - \frac{1}{3} < i.$$
The mesas were indexed to be in increasing order, so $|M_{m_i}| = i$. Thus we have contradicted Theorem~\ref{thm:characterization} and so there can be no such $i$.
\end{proof}

The tri-cyclic nature of mesa set properties gives special relevance to $|\AMS_{3k-1}|$, whose elements, by Lemma~\ref{L:most}, can achieve the highest proportion of mesas among their values. The particular merit of this consideration will become apparent in Theorem~\ref{thm:3k-recursion}.

\section{Enumerations}\label{sec:enumerations}

We now focus on the enumeration of admissible mesa sets, and we will show that the sequence $\{\left|\AMS_n\right|\}$ obeys several recurrence relations. We preface those results with two useful lemmas. 

\begin{lem}\label{lem:ams of order n vs of order n+1} 
For $n\geq 1$, the following are true.\ 
\begin{enumerate}\renewcommand{\labelenumi}{(\alph{enumi})}
    \item $\AMS_n \subset \AMS_{n+1}$. 
    \item 
    For any $M \in \left(\AMS_{n+1} \setminus \AMS_n\right)$, we have $\left(M \setminus \{n+1\}\right) \in \AMS_n$, and hence the map 
    \begin{equation}\label{eqn:map into AMS_n}
    \varphi_n : (\AMS_{n+1} \setminus \AMS_n) \to \AMS_n
    \end{equation}
    defined by $\varphi_n(M) \coloneqq M \setminus \{n+1\}$ is injective. 
\end{enumerate}
\end{lem}

\begin{proof} \ 
  \begin{enumerate}\renewcommand{\labelenumi}{(\alph{enumi})}
    \item This is a special case of an earlier observation. In particular, let $w$ be the canonical witness Stirling permutation for some $M \in \AMS_n$. Append two copies of $n+1$ to the right end of $w$ to produce a Stirling permutation of order $n+1$ whose mesa set is $M$.

    \item First note that $n+1$ must belong to all elements of $\AMS_{n+1} \setminus \AMS_n$. Now consider $M \in \AMS_{n+1}$ with $n+1 \in M$, and  let $w$ be the canonical witness for $M$. Deleting the two copies of $n+1$ from $w$ produces a Stirling permutation of order $n$, whose mesa set is precisely $M \setminus \{n+1\}$. The injectivity of the described map follows immediately.\qedhere
  \end{enumerate}
\end{proof}

From this result, it follows that understanding the set difference $\AMS_{n+1} \setminus \AMS_n$ amounts to determining whether the map from Lemma~\ref{lem:ams of order n vs of order n+1}(b) is surjective, or in what manner it fails to be surjective. The map $\varphi_n$, defined in \eqref{eqn:map into AMS_n}, will be useful throughout this section.

\begin{lem}\label{lem:when we can insert a big value}
Suppose that $M \in \AMS_n$ for some $n$, where $M$ is not of maximal size; that is, $|M| < \lfloor (2n-1)/3\rfloor$. Then $M \cup \{n+1\} \in \AMS_{n+1}$.
\end{lem}

\begin{proof}
Consider the canonical witness Stirling permutation $w$ for $M$. Since $|M|=l < \lfloor (2n-1)/3\rfloor$, we have
$$w = t_1 \ m_1 \ m_1 \ t_2 \ m_2 \ m_2 \ \cdots \ t_l \ m_l \ m_l \ t_{l+1}\ \cdots \ t_{2(n-l)} \ t_{2(n-l)},
$$
maintaining the notation from Definition~\ref{defn:canonical witness}. Therefore, 
$$t_1 \ m_1 \ m_1 \ t_2 \ m_2 \ m_2 \ \cdots \ t_l \ m_l \ m_l \ t_{l+1}\ \cdots \ t_{2(n-l)} \ (n+1) \ (n+1) \ t_{2(n-l)}$$
is a Stirling permutation of order $n+1$, and its mesa set is $M \cup \{n+1\}$.
\end{proof}

We now present the first two recurrence relations for the sequence $\{\left|\AMS_n\right|\}$. 

\begin{thm}\label{thm:relation}
If $k\geq 1$, then $\left|\AMS_{3k+1}\right|=2\left|\AMS_{3k}\right|$ and $\left|\AMS_{3k-1}\right|=2 \left|\AMS_{3k-2}\right|$. 
\end{thm}

\begin{proof} 
Due to Lemma~\ref{lem:ams of order n vs of order n+1}, it remains to understand the elements of $\AMS_{3k+1} \setminus \AMS_{3k}$ and of $\AMS_{3k-1} \setminus \AMS_{3k-2}$; that is, the admissible mesa sets containing $3k+1$ and $3k-1$, respectively.

By Lemma~\ref{L:most}, each $M\in \AMS_{3k}$ satisfies $|M|\le 2k-1$. Since 
$$2k-1 < \left\lfloor \frac{2(3k+1)-1}{3}\right\rfloor,$$
each $M \in \AMS_{3k}$ appears in $\AMS_{3k+1}$ as an admissible mesa set of non-maximal size. Thus, by Lemma~\ref{lem:when we can insert a big value}, $M \cup \{3k+1\} \in \AMS_{3k+1}$. Therefore, the map $\varphi_{3k}:\AMS_{3k+1} \setminus \AMS_{3k} \to \AMS_{3k}$ underlying Lemma~\ref{lem:ams of order n vs of order n+1}(b) is also surjective. Hence, 
$\left|\AMS_{3k+1} \setminus \AMS_{3k}\right| = \left|\AMS_{3k}\right|$
and consequently $\left|\AMS_{3k+1}\right| = 2\left|\AMS_{3k}\right|$.

Similarly, Lemma~\ref{L:most} says that if $M\in \AMS_{3k-2}$, then $|M|\le 2k-2$, and 
$$2k-2 < \lfloor(2(3k-1)-1)/3\rfloor.$$
Hence an analogous argument shows that $\left|\AMS_{3k-1}\right|= 2\left|\AMS_{3k-2}\right|$. 
\end{proof}

These doubling results do not hold for the relationship between $|\AMS_{3k}|$ and $|\AMS_{3k-1}|$. Indeed, it is easy to see that $\varphi_{3k-1}$ is not surjective, meaning that $\left|\AMS_{3k}\right|<2\left|\AMS_{3k-1}\right|$.

\begin{ex}
By Lemma~\ref{L:most}, Stirling permutations of order $7$ have at most $3$ pinnacles, as do Stirling permutations of order $6$. Thus, for example, the maximal mesa set $M = \lbrace 3,4,5 \rbrace \in \AMS_{5}$ is not in the image of the map $\varphi_5$ because $\{3,4,5,6\} \not\in \AMS_6$. 
\end{ex}

More generally, Stirling permutations of order $3k$ have at most $2k - 1$ mesas, as do Stirling permutations of order $3k-1$. Therefore, the image of $\varphi_{3k-1}$ is missing all elements of $\AMS_{3k-1}$ that have cardinality $2k-1$.

\begin{lem}\label{lem:can't add to a full set}
    Consider $M\in \AMS_{3k-1}$. Then $M \cup \{3k\} \notin \AMS_{3k}$ if and only if $|M|=2k-1$.
\end{lem}

\begin{proof}
    Suppose that $M \in \AMS_{3k-1}$ with $M \cup \{3k\} \not\in \AMS_{3k}$. Write 
    $\overline{M}\coloneqq[3k-1]\setminus M =\{u_1<u_2<\cdots\}$ as in Definition~\ref{defn:canonical witness}. 
      
    If $|M|=l < 2k-1$ then by Lemma~\ref{L:most} it follows that $M$ is not maximal and thus has a canonical witness permutation of the form
\[
t_1 \ m_1 \ m_1 \ t_2 \ m_2 \ m_2 \ \cdots \ t_l \ m_l \ m_l \ t_{l+1} \ \cdots \ t_{2(3k-1-l)},
\]
where  $t_j=u_{\lceil j/2\rceil}$
and $l+1 < 2(3k-1-l)$. Inserting two copies of the value $3k$ between $t_{l+1}$ and $t_{l+2}$ yields the Stirling permutation
\[
t_1 \ m_1 \ m_1 \ t_2 \ m_2 \ m_2 \ \cdots \ t_l \ m_l \ m_l \ t_{l+1} \ (3k) \ (3k) \ t_{l+2} \ \cdots \ t_{2(3k-1-l)}
\]
of order $3k$, with mesa set $M \cup \{3k\}$, and hence $M \cup \{3k\} \in \AMS_{3k}$. This is a contradiction and hence, $|M| = 2k-1$. 

For the converse, suppose that $M \in \AMS_{3k-1}$ with $|M| = 2k-1$. In particular, since $M \in \AMS_{3k-1}$ we have $3k \notin M$. Thus $|M \cup \{ 3k \}| = 2k$ which by Lemma~\ref{L:most} implies that $M \cup \{ 3k \} \notin \AMS_{3k}$.
\end{proof}

We can now give the final recursive relation for $\{|\AMS_n|\}$, continuing the work of Theorem~\ref{thm:relation}.

\begin{thm}\label{thm:3k-recursion}
If $k \ge 1$, then $|\AMS_{3k}| = 2|\AMS_{3k-1}| - |\lbrace M \in \AMS_{3k-1} : |M| = 2k-1 \rbrace|$.
\end{thm}

\begin{proof}
Combining the characterization of Lemma~\ref{lem:can't add to a full set} with the injectivity of the map $\varphi_{3k-1}: \AMS_{3k} \setminus \AMS_{3k-1} \to \AMS_{3k-1}$ from \eqref{eqn:map into AMS_n}, we obtain
$$|\AMS_{3k} \setminus \AMS_{3k-1}| + |\lbrace M \in \AMS_{3k-1} : |M| = 2k-1 \rbrace| = | \AMS_{3k-1}|,$$
from which the theorem statement follows. 
\end{proof}

In summary, for $k\geq 1$, Theorems~\ref{thm:relation} and~\ref{thm:3k-recursion} yield:
\begin{align*}
    |\AMS_{3k-1}| &= 2|\AMS_{3k-2}|,\\
    |\AMS_{3k}| &= 2|\AMS_{3k-1}| - |\lbrace M \in \AMS_{3k-1} : |M| = 2k-1 \rbrace|, \text{ and}\\
    |\AMS_{3k+1}| &= 2|\AMS_{3k}|.
\end{align*}    

We show in Theorem~\ref{thm:dyck paths correspond to 2k-1 mesas in order 3k-1} that $|\lbrace M \in \AMS_{3k-1} : |M| = 2k-1 \rbrace|$ is counted by \cite[\href{https://oeis.org/A006013}{A006013}]{OEIS}. From these results and the initial value $|\AMS_1| = 1$ (since $\AMS_1 = \{\emptyset\}$), we can compute $|\AMS_n|$ recursively for any $n$. 
Initial data is presented in Table~\ref{table:ams}, and in Corollary~\ref{cor:complete enumeration} we provide explicit formulas for this data.

\begin{table}[htbp]
{\renewcommand{\arraystretch}{2}$\begin{array}{c||c|c|c|c|c|c|c|c|c|c|c|c|c|c|c}
n & 1 & 2 & 3 & 4 & 5 & 6 & 7 & 8 & 9 & 10 & 11 & 12 & 13 & 14 & 15\\
\hline
\hline
\left\vert \AMS_n \right\vert & 1 & 2 & 3 & 6 & 12 & 22 & 44 & 88 & 169 & 338 & 676 & 1322 & 2644 & 5288 & 10433
\end{array}$}
\vspace{.1in}
\caption{The number of admissible mesa sets in Stirling permutations of order $n$, for $1 \le n \le 15$.}\label{table:ams}
\end{table}

As defined by Armstrong et.~al \cite{Armstrong}, 
an $(m,n)$-\emph{Dyck path} is lattice path $\pi$ consisting of north and east steps from $(0,0)$ to $(n,m)$ that remains below the line $y=\frac{m}{n}x$. 
Given coprime positive integers $m$ and $n$, it is well known that the number of $(m,n)$-Dyck paths is given by the \emph{rational Catalan number} \cite[\href{https://oeis.org/A328901}{A328901}]{OEIS}:
\[C_{m,n} = \frac{1}{n}\binom{n+m-1}{m}.\]
For the remainder of this section, we consider mesa sets $M \in \AMS_{3k-1}$ with $|M| = 2k-1$. These are witnessed by permutations for which the remaining $(3k-1)-(2k-1) = k$ non-mesa values are necessarily local minima. It turns out these maximal mesa sets are in bijection with the set of $(2k-1,k)$-Dyck paths as follows. 

\begin{defin}\label{defn:dyck path map}
    Define the map $\delta$ from $\{M \in \AMS_{3k-1} : |M| = 2k-1\}$ to lattice paths from  $(0,0)$ to $(k,2k-1)$ by
    $$\delta : M \mapsto \pi_1 \cdots \pi_{3k-1},$$
    where
    $$\pi_i \coloneqq \begin{cases}
        \text{north} & \text{ if } i \in M, \text{ and}\\
        \text{east} & \text{ otherwise}.
    \end{cases}$$
\end{defin}

Note that the choice of domain for $\delta$ means that the east steps in $\delta(M)$ correspond exactly to the local minima in a $w \in Q_{3k-1}$ for which $\Mesa(w) = M$.

\begin{ex} 
Consider $k=3$ and $M = \{3,5,6,7,8\}$. Note that $M \in \AMS_8$, and has maximal size, since $1331552662774884$ is a witness permutation for $M$ in $Q_8$. In particular, as seen in the second image in Figure~\ref{fig:dyckbijection}, the path $\delta(M)$ corresponds to the $(5,3)$-lattice path whose $i$th step is north if $i \in M$, and east otherwise. 
\end{ex}

\begin{figure}
    \centering
\begin{tikzpicture}[scale=.7]
% pick an m and n HERE
  \pgfmathtruncatemacro{\n}{3} %i.e here I set n=6
  \pgfmathtruncatemacro{\m}{2*\n-1} %i.e here I set m=5
%%%%%%% DONT TOUCH THIS PART %%%%%%%%%%%%%%%%
 \pgfmathtruncatemacro{\mi}{\m-1} %\mi=m-1 
    \pgfmathtruncatemacro{\ni}{\n-1} %\ni=n-1
    %
%generates the grid and anderson labels
\draw[dashed, thick, red] (0,0)--(\n,\m); %draws the diagonal from (0,m) -- (n,0)
  \foreach \x in {0,...,\n} % n=5
    \foreach \y  [count=\yi] in {0,...,\mi}   % m-1=6  
    %\draw (\x\y)--(\x\yi);
      \draw (\x,\y)--(\x,\yi);
        \foreach \y in {0,...,\m}     % m=7  
       \foreach \x  [count=\xi] in {0,...,\ni}   % n-1=4
        %\draw (\x\y)--(\xi\y);
         \draw (\x,\y)--(\xi,\y);
%%%%%%%%%%%%%%%%%%%%%%%%%%%%%%%%%%%%%%%%%%%%%%%%%
\draw [line width=2pt] (0,0)--(3,0)--(3,5);
\node at (1.5,-1){$\{4,5,6,7,8\}$};
\end{tikzpicture}
\qquad\qquad
\begin{tikzpicture}[scale=.7]
% pick an m and n HERE
  \pgfmathtruncatemacro{\n}{3} %i.e here I set n=6
  \pgfmathtruncatemacro{\m}{2*\n-1} %i.e here I set m=5
%%%%%%% DONT TOUCH THIS PART %%%%%%%%%%%%%%%%
 \pgfmathtruncatemacro{\mi}{\m-1} %\mi=m-1 
    \pgfmathtruncatemacro{\ni}{\n-1} %\ni=n-1
    %
%generates the grid and anderson labels
\draw[dashed, thick, red] (0,0)--(\n,\m); %draws the diagonal from (0,m) -- (n,0)
  \foreach \x in {0,...,\n} % n=5
    \foreach \y  [count=\yi] in {0,...,\mi}   % m-1=6  
    %\draw (\x\y)--(\x\yi);
      \draw (\x,\y)--(\x,\yi);
        \foreach \y in {0,...,\m}     % m=7  
       \foreach \x  [count=\xi] in {0,...,\ni}   % n-1=4
        %\draw (\x\y)--(\xi\y);
         \draw (\x,\y)--(\xi,\y);
%%%%%%%%%%%%%%%%%%%%%%%%%%%%%%%%%%%%%%%%%%%%%%%%%
\draw [line width=2pt] (0,0)--(2,0)--(2,1)--(3,1)--(3,5);
\node at (1.5,-1){$\{3,5,6,7,8\}$};
\end{tikzpicture}
\qquad\qquad
\begin{tikzpicture}[scale=.7]
% pick an m and n HERE
  \pgfmathtruncatemacro{\n}{3} %i.e here I set n=6
  \pgfmathtruncatemacro{\m}{2*\n-1} %i.e here I set m=5
%%%%%%% DONT TOUCH THIS PART %%%%%%%%%%%%%%%%
 \pgfmathtruncatemacro{\mi}{\m-1} %\mi=m-1 
    \pgfmathtruncatemacro{\ni}{\n-1} %\ni=n-1
    %
%generates the grid and anderson labels
\draw[dashed, thick, red] (0,0)--(\n,\m); %draws the diagonal from (0,m) -- (n,0)
  \foreach \x in {0,...,\n} % n=5
    \foreach \y  [count=\yi] in {0,...,\mi}   % m-1=6  
    %\draw (\x\y)--(\x\yi);
      \draw (\x,\y)--(\x,\yi);
        \foreach \y in {0,...,\m}     % m=7  
       \foreach \x  [count=\xi] in {0,...,\ni}   % n-1=4
        %\draw (\x\y)--(\xi\y);
         \draw (\x,\y)--(\xi,\y);
%%%%%%%%%%%%%%%%%%%%%%%%%%%%%%%%%%%%%%%%%%%%%%%%%
\draw [line width=2pt] (0,0)--(2,0)--(2,2)--(3,2)--(3,5);
\node at (1.5,-1){$\{3,4,6,7,8\}$};
\end{tikzpicture}
\qquad\qquad
\begin{tikzpicture}[scale=.7]
% pick an m and n HERE
  \pgfmathtruncatemacro{\n}{3} %i.e here I set n=6
  \pgfmathtruncatemacro{\m}{2*\n-1} %i.e here I set m=5
%%%%%%% DONT TOUCH THIS PART %%%%%%%%%%%%%%%%
 \pgfmathtruncatemacro{\mi}{\m-1} %\mi=m-1 
    \pgfmathtruncatemacro{\ni}{\n-1} %\ni=n-1
    %
%generates the grid and anderson labels
\draw[dashed, thick, red] (0,0)--(\n,\m); %draws the diagonal from (0,m) -- (n,0)
  \foreach \x in {0,...,\n} % n=5
    \foreach \y  [count=\yi] in {0,...,\mi}   % m-1=6  
    %\draw (\x\y)--(\x\yi);
      \draw (\x,\y)--(\x,\yi);
        \foreach \y in {0,...,\m}     % m=7  
       \foreach \x  [count=\xi] in {0,...,\ni}   % n-1=4
        %\draw (\x\y)--(\xi\y);
         \draw (\x,\y)--(\xi,\y);
%%%%%%%%%%%%%%%%%%%%%%%%%%%%%%%%%%%%%%%%%%%%%%%%%
\draw [line width=2pt] (0,0)--(2,0)--(2,3)--(3,3)--(3,5);
\node at (1.5,-1){$\{3,4,5,7,8\}$};
\end{tikzpicture}\\ 

\smallskip

\bigskip
\begin{tikzpicture}[scale=.7]
% pick an m and n HERE
  \pgfmathtruncatemacro{\n}{3} %i.e here I set n=6
  \pgfmathtruncatemacro{\m}{2*\n-1} %i.e here I set m=5
%%%%%%% DONT TOUCH THIS PART %%%%%%%%%%%%%%%%
 \pgfmathtruncatemacro{\mi}{\m-1} %\mi=m-1 
    \pgfmathtruncatemacro{\ni}{\n-1} %\ni=n-1
    %
%generates the grid and anderson labels
\draw[dashed, thick, red] (0,0)--(\n,\m); %draws the diagonal from (0,m) -- (n,0)
  \foreach \x in {0,...,\n} % n=5
    \foreach \y  [count=\yi] in {0,...,\mi}   % m-1=6  
    %\draw (\x\y)--(\x\yi);
      \draw (\x,\y)--(\x,\yi);
        \foreach \y in {0,...,\m}     % m=7  
       \foreach \x  [count=\xi] in {0,...,\ni}   % n-1=4
        %\draw (\x\y)--(\xi\y);
         \draw (\x,\y)--(\xi,\y);
%%%%%%%%%%%%%%%%%%%%%%%%%%%%%%%%%%%%%%%%%%%%%%%%%
\draw [line width=2pt] (0,0)--(1,0)--(1,1)--(3,1)--(3,5);
\node at (1.5,-1){$\{2,5,6,7,8\}$};
\end{tikzpicture}
\qquad\qquad
\begin{tikzpicture}[scale=.7]
% pick an m and n HERE
  \pgfmathtruncatemacro{\n}{3} %i.e here I set n=6
  \pgfmathtruncatemacro{\m}{2*\n-1} %i.e here I set m=5
%%%%%%% DONT TOUCH THIS PART %%%%%%%%%%%%%%%%
 \pgfmathtruncatemacro{\mi}{\m-1} %\mi=m-1 
    \pgfmathtruncatemacro{\ni}{\n-1} %\ni=n-1
    %
%generates the grid and anderson labels
\draw[dashed, thick, red] (0,0)--(\n,\m); %draws the diagonal from (0,m) -- (n,0)
  \foreach \x in {0,...,\n} % n=5
    \foreach \y  [count=\yi] in {0,...,\mi}   % m-1=6  
    %\draw (\x\y)--(\x\yi);
      \draw (\x,\y)--(\x,\yi);
        \foreach \y in {0,...,\m}     % m=7  
       \foreach \x  [count=\xi] in {0,...,\ni}   % n-1=4
        %\draw (\x\y)--(\xi\y);
         \draw (\x,\y)--(\xi,\y);
%%%%%%%%%%%%%%%%%%%%%%%%%%%%%%%%%%%%%%%%%%%%%%%%%
\draw [line width=2pt] (0,0)--(1,0)--(1,1)--(2,1)--(2,2)--(3,2)--(3,5);
\node at (1.5,-1){$\{2,4,6,7,8\}$};
\end{tikzpicture}
\qquad\qquad
\begin{tikzpicture}[scale=.7]
% pick an m and n HERE
  \pgfmathtruncatemacro{\n}{3} %i.e here I set n=6
  \pgfmathtruncatemacro{\m}{2*\n-1} %i.e here I set m=5
%%%%%%% DONT TOUCH THIS PART %%%%%%%%%%%%%%%%
 \pgfmathtruncatemacro{\mi}{\m-1} %\mi=m-1 
    \pgfmathtruncatemacro{\ni}{\n-1} %\ni=n-1
    %
%generates the grid and anderson labels
\draw[dashed, thick, red] (0,0)--(\n,\m); %draws the diagonal from (0,m) -- (n,0)
  \foreach \x in {0,...,\n} % n=5
    \foreach \y  [count=\yi] in {0,...,\mi}   % m-1=6  
    %\draw (\x\y)--(\x\yi);
      \draw (\x,\y)--(\x,\yi);
        \foreach \y in {0,...,\m}     % m=7  
       \foreach \x  [count=\xi] in {0,...,\ni}   % n-1=4
        %\draw (\x\y)--(\xi\y);
         \draw (\x,\y)--(\xi,\y);
%%%%%%%%%%%%%%%%%%%%%%%%%%%%%%%%%%%%%%%%%%%%%%%%%
\draw [line width=2pt] (0,0)--(1,0)--(1,1)--(2,1)--(2,3)--(3,3)--(3,5);
\node at (1.5,-1){$\{2,4,5,7,8\}$};
\end{tikzpicture}    
  \caption{All possible maximal mesa sets $M \in \AMS_{8}$ and their corresponding $(5,3)$-Dyck paths under the map $\delta$ represented above them.} \label{fig:dyckbijection}
\end{figure}
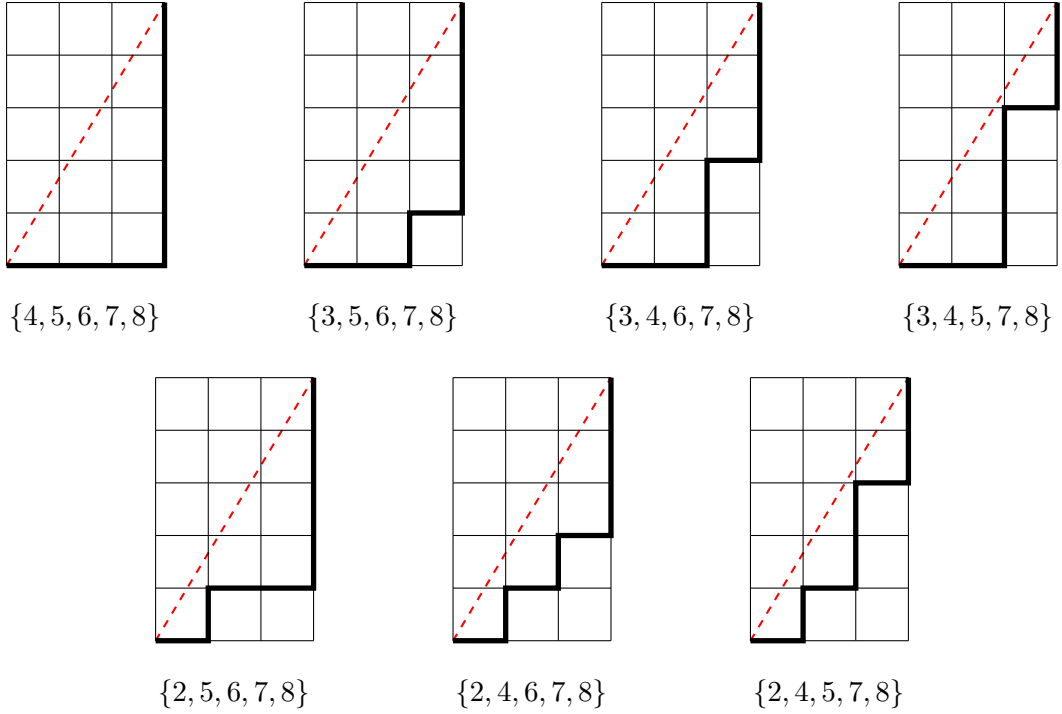

\begin{thm}\label{thm:dyck paths correspond to 2k-1 mesas in order 3k-1}
    The map $\delta$ is a bijection from $\{M \in \AMS_{3k-1} : |M| = 2k-1\}$ to the set of $(2k-1,k)$-Dyck paths.
\end{thm}

\begin{proof}
    Fix $M \in \AMS_{3k-1}$ with $|M| = 2k-1$. Then $\delta(M)$ has $2k-1$ north steps and $k$ east steps, and thus $\delta(M)$ describes a path from $(0,0)$ to $(k,2k-1)$. Certainly, the map $\delta$ is injective, so it remains to show that its image is the set of $(2k-1,k)$-Dyck paths.

    Let $\mathcal{L}$ be the line $y = x \cdot (2k-1)/k$. Consider any lattice point $(a,b) \in [0,k] \times [0,2k-1]$ with $(a,b)\neq (0,0)$, $(a,b)\neq (k,2k-1)$. This point $(a,b)$ lies below $\mathcal{L}$ if and only if
    $$b < a \cdot \frac{2k-1}{k} = 2a - \frac{a}{k}.$$
    Because $a \in [0,k]$, we see that, in fact, $(a,b)$ lies below $\mathcal{L}$ if and only if
    $$b \le 2a - 1.$$ 
    Since the north steps represent mesas and the east steps represent local minima, this inequality recovers the inequality of Theorem~\ref{thm:characterization}. In other words, a path $\delta(M)$ passes through $(a,b)$ if and only if $(a,b)$ lies below the line $y = x \cdot (2k-1)/k$. Thus $\delta(M)$ is a $(2k-1,k)$-Dyck path. The converse, that any $(2k-1,k)$-Dyck path is the image under $\delta$ of some $M \in \AMS_{3k-1}$ with $|M| = 2k-1$ is immediate from the construction since any Dyck path is uniquely determined by the position of its north steps. Namely, if $\pi=\pi_1\cdots \pi_{3k-1}$ is a $(2k-1,k)$-Dyck path, then the corresponding admissible mesa set consists of the indices of the north steps of $\pi$.
\end{proof}

Since $k$ and $2k-1$ are relatively prime, we instantly obtain the following enumeration.

\begin{cor}\label{cor:counting maximal mesa sets in 3k-1}
For any $k \ge 1$,
$$|\{M \in \AMS_{3k-1} : |M| = 2k-1\}| = C_{2k-1,k} = \frac{1}{k}\binom{3k-2}{2k-1}.$$
\end{cor}

\begin{ex}
In Figure~\ref{fig:dyckbijection} all maximal admissible mesa sets for $Q_8$ and their corresponding $(5,3)$-Dyck paths are listed. Indeed, as predicted by Corollary~\ref{cor:counting maximal mesa sets in 3k-1}, there are precisely $C_{5,3}=7$ admissible mesa sets.
\end{ex}

As suggested previously, another immediate consequence of Theorem~\ref{thm:dyck paths correspond to 2k-1 mesas in order 3k-1} is that  $|\{M\in \AMS_{3k-1} \ : \ |M|=2k-1\}|$ coincides with \cite[\href{https://oeis.org/A006013}{A006013}]{OEIS}, setting $n\coloneqq k-1$.

We conclude this section by giving an explicit formula for $|\AMS_n|$, for all $n$.

\begin{cor}\label{cor:complete enumeration}
Let $n = 3k + r$ where $r \in \{0,1,2\}$ and $k$ is a nonnegative integer. Then
$$|\AMS_n| = 2^{n-1} - \sum_{i=0}^{k-1}2^{3i+r}C_{2(k-i)-1, k-i}.$$
\end{cor}

\begin{proof}
 Following Theorems~\ref{thm:relation} and~\ref{thm:3k-recursion} and Corollary~\ref{cor:counting maximal mesa sets in 3k-1}, we have
\begin{align*}
    \left\vert \AMS_{3k}\right\vert &= 2\left\vert\AMS_{3k-1}\right\vert - C_{2k-1,k}\\
    &=2^3\left\vert\AMS_{3k-3}\right\vert - C_{2k-1,k}.
\end{align*}
Recursively applying this procedure and noting that $|\AMS_2| = 2$ yields
\begin{align*}
    \left\vert \AMS_{3k}\right\vert &= 2^{3k-3}\left\vert\AMS_3\right\vert - \sum_{i=0}^{k-2}2^{3i} C_{2(k-i)-1,k-i}\\
    &= 2^{3k-3} \left(2\left\vert\AMS_2\right\vert - C_{1,1}\right) - \sum_{i=0}^{k-2}2^{3i} C_{2(k-i)-1,k-i}\\
    &= 2^{3k-1} - \sum_{i=0}^{k-1} 2^{3i} C_{2(k-i)-1,k-i}.
\end{align*}
The enumerations of $\AMS_{3k\pm 1}$ follow analogously.
\end{proof}

The values $|\AMS_n|$ appear in Table~\ref{table:ams} for small values of $n$ and as \cite[\href{https://oeis.org/A363582}{A363582}]{OEIS}.

\section{Future work}\label{sec:future}

A direct consequence of our work is the enumeration of admissible mesa sets of maximal size for any $n$. Analogous to the bijection in Definition~\ref{defn:dyck path map}, by considering the maps $\varphi_{3k+1}: \AMS_{3k+2}\setminus \AMS_{3k+1} \to \AMS_{3k+1}$ and $\varphi_{3k}: \AMS_{3k+1}\setminus \AMS_{3k} \to \AMS_{3k}$ from Lemma~\ref{lem:ams of order n vs of order n+1} it is easy to see that similar bijections can be constructed from the set of maximal pinnacle sets having $n=3k+1$ (resp., $n=3k$) to lattice paths terminating at position $(k,2k)$ (resp., $(k,2k-1)$) which can be trivially extended to bijections into $(2k+1,k+1)$-Dyck paths, since the final two steps of such Dyck paths are always northward. 
For instance, under these bijections the mesa sets $\{3,5,6\} \in Q_6$ and $\{3,5,6,7\} \in Q_7$ would map to the same Dyck path as $\{ 3,5,6,7,8\} \in Q_8$, seen in Figure~\ref{fig:dyckbijection}. In particular, these maps provide an alternative way to derive the enumeration results of Corollary~\ref{cor:complete enumeration}. Consequently, another interesting direction would be to enumerate distinct admissible mesa sets based on their cardinalities.

There are many potential ways to leverage the association between a general admissible mesa set and a Dyck path. A natural question is whether known Dyck path statistics, such as \emph{area} or \emph{dinv}, relevant in the study of parking functions and $(q,t)$-Catalan combinatorics, can shed light on the study of mesa sets for Stirling permutations, and vice-versa. Indeed, it is easy to see that given any maximal mesa set $M \in \AMS_{3k-1}$, the area of $\delta(M)$ is precisely the number of \emph{inversions} between $M$ and $\overline{M} = [3k-1]\setminus M$; that is, the number of pairs $(m,u)$ with $m \in M$ and $u \in \overline{M}$  such that $m<u$. For example, in Figure~\ref{fig:dyckbijection} we can see that for $M=\{2,4,6,7,8\}$ and $\overline{M}=\{1,3,5\}$ we have area$(\delta(M))=3$ and, indeed, because $2<3$, $2<5$, and $4<5$, then $M$ has $3$ inversions with respect to $\overline{M}$. While a quick computation will show that dinv counts some subset of these inversions, it is unclear what the combinatorial meaning of this statistic is in the context of Stirling permutations.

Finally, recall that Mesa sets are the generalization of permutations' pinnacle sets to the setting of Stirling permutations. Thus, one might also consider exploring the analogue of peaks in the setting of Stirling permutations, where a peak no longer occurs at a single index, but rather at two adjacent indices. Similarly, just as one can attempt to enumerate permutations with given peak or pinnacle sets, the enumeration of Stirling permutations having a fixed mesa set remains an interesting open question.

\subsection*{Acknowledgements}

The authors thank Patrek K\'arason Ragnarsson for the coding and data that facilitated the research in this project.

\bibliographystyle{plain}
\bibliography{Bibliography.bib}

\end{document}